\documentclass[10pt,reqno]{amsart}
\usepackage[utf8]{inputenc}
\usepackage[T1]{fontenc}
\usepackage[english]{babel}
\usepackage[margin=3.5cm]{geometry}
\usepackage{amsmath, amsfonts, amsthm, amssymb, graphicx,mathrsfs}
\usepackage[usenames,dvipsnames,svgnames,table]{xcolor}
\usepackage[colorlinks=true, hyperindex, pagebackref=true, citecolor=magenta,linkcolor=Blue]{hyperref}
\usepackage[all]{xy}
\usepackage{enumitem}
\setlist[enumerate]{label=\rm(\arabic*)}
\usepackage[symbol]{footmisc}
\usepackage{shadethm}

\newtheorem{lemma}[subsection]{Lemma}
\newtheorem{proposition}[subsection]{Proposition}
\newtheorem{theorem}[subsection]{Theorem}
\newtheorem{remark}[subsection]{Remark}
\newtheorem{definition}[subsection]{Definition}
\newtheorem{corollary}[subsection]{Corollary}
\newtheorem{example}[subsection]{Example}

\newcommand{\bbR}{\mathbb{R}}

\newcommand{\bbC}{\mathbb{C}}
\newcommand{\bbN}{\mathbb{N}}
\newcommand{\bfF}{\mathbf{F}}

\newcommand{\F}{\mathcal{F}(M,P)}
\newcommand{\FF}{\mathcal{F}^{\circ}(M,P)}

\newcommand{\id}{\mathrm{id}}
\newcommand{\Crit}{\Sigma_f}
\newcommand{\CritP}{\Crit^p}
\newcommand{\CritC}{\Crit^C}

\newcommand{\fz}{f_0}

\newcommand{\per}{\mathsf{per}}
\newcommand{\sst}{\sigma^0}

\newcommand{\CYL}{S^1\times [0,1]}
\newcommand{\DIS}{D^2}
\newcommand{\SPH}{S^2}
\newcommand{\TOR}{T^2}

\newcommand{\xsurj}{\ar@{->>}[r]}
\newcommand{\xinj}{\ar@{^{(}->}[r]}

\title[Normal forms for functions on surfaces ]{Normal forms of functions with degenerate singularities on surfaces equipped with semi-free circle actions}
\author{Bohdan Feshchenko}
\address{Department of algebra and topology, Institute of Mathematics of National Academy of Science of Ukraine,
	Tereshchenkivska, 3, Kyiv, 01601, Ukraine}
\email{fb@imath.kiev.ua}
\date{\today}
\begin{document}
	\begin{abstract}
		This article is devoted to the study of a certain class of smooth circle-valued functions on a cylinder $S^1\times [0,1]$, a torus $T^2$, a disk $D^2$ and a sphere $S^2$ which is a generalization of  Morse-Bott functions without saddles. We established a "normal form" for functions from this class, namely, we proved that any such function represents in the form $f = \varkappa\circ f_0\circ h^{-1}$, where $f_0$ is the simplest Morse function on the given surface for some diffeomorphism $h$ of $M$ and a smooth function $\varkappa$ satisfying some natural conditions.
	\end{abstract}	
	\maketitle

\section{Introduction}
The study of smooth functions on manifolds and their properties lies at the heart  of a modern differential topology and a theory of dynamical systems. It connects local geometric properties of functions to the global topology of the underlying surfaces. This connection is well understood for generic classes of smooth functions like Morse functions, functions with isolated singularities, Morse-Bott functions. 

Let $M$ be a smooth manifold, and $P$ be either a real line $\bbR$ or a circle $S^1$.
Recall that two smooth $P$-valued function $f,g:M\to P$ are called topologically (smoothly) equivalent, if there exist homeomorphisms (diffeomorphisms) $h:M\to M$ and $\ell:P\to P$ such that 
\begin{equation}\label{eq:f-g-def}
f = \ell\circ g\circ h^{-1}.
\end{equation}

The question of when two functions are smoothly or topologically equivalent is challenging and remains open in general. Topological equivalence of functions with isolated singularities on smooth compact surfaces were studied by V.~Sharko \cite{Sharko:UMZ:2003}, E.~V.~Kulinich \cite{Kulinich:1998:MFAT}, A.~Prishlyak \cite{Prishlyak:TA:2002},
 A.~Kadubovskiy  \cite{Kadubovskiy:UMZH:2006}, for simple Morse functions by J.~Mart\'{\i}nez-Alfaro,  I.~S.~Meza-Sarmiento,  and R.~Oliveira
 \cite{MartinezMezaOliveria:2016}, for circle-valued simple Morse-Bott functions were studied by E.~B.~Batista, J.~C.~F.~Costa and I.~S.~Meza-Sarmiento \cite{BatistaCostaMeza:2022}.
 V.~Sharko obtained the description of connected components of the space of Morse functions.

The functions (Definition \ref{def:class-FF}) we will study in this paper are closely related to questions of the study of homotopy properties stabilizers of smooth functions. 
We will recall the necessary definitions.
 The product $\mathcal{D}(P)\times \mathcal{D}(M)$ of the groups of diffeomorphisms of $P$ and diffeomorphisms of $M$ naturally acts from the left on the space of smooth $P$-valued functions $C^{\infty}(M,P)$ by the rule:
$$
\gamma_{LR}: \mathcal{D}(P)\times \mathcal{D}(M)\times  C^{\infty}(M,P)\to C^{\infty}(M,P),\qquad \gamma_{LR}(\ell,h, f) = \ell\circ f\circ h^{-1}.
$$
The group $\mathcal{D}(M) = \id_M \times \mathcal{D}(M)$ is a subgroup of 
$\mathcal{D}(P)\times \mathcal{D}(M)$ and $\gamma_{LR}$ induces a left action
$$
\gamma_R: \mathcal{D}(M)\times  C^{\infty}(M,P)\to C^{\infty}(M,P),\qquad \gamma_{R}(h, f) = f\circ h^{-1}.
$$
Actions $\gamma_{LR}$ and $\gamma_R$ are left actions, but
it will be convenient to
use the terms “left-right” and “right” to refer the sides at which we apply the corresponding diffeomorphisms.
Note that two smooth functions $f,g:M\to P$ are smooth equivalent, if they belongs to the same orbit of the action $\gamma_{LR}.$

\subsection{Stabilizers of functions from $C^{\infty}(M,P)$}
Denote by $\mathcal{D}^+(P)$ the group of orientation preserving diffeomorphisms of $P.$
Let $f$ be a smooth function in $C^{\infty}(M,P)$. The image $\mathrm{Im}(f) = f(M)$ is a smooth submanifold of $P$.
Denote by $\mathcal{D}^+(\mathrm{Im}(f))$ a subgroup of  $\mathcal{D}^+(P)$.
One can define stabilizers of $f$ with respect to the restriction of the action $\gamma_{LR}$ on $\mathcal{D}^+(\mathrm{Im}(f))\times \mathcal{D}(M)$ and the action $\gamma_R$ by the formulas:
\begin{align*}
\mathcal{S}_{LR}(f) &= \{ (\ell, h)\in \mathcal{D}^+(\mathrm{Im}(f))\times \mathcal{D}(M)\,|\, \ell\circ f\circ h^{-1} = f  \},\\
\mathcal{S}_R(f) &= \{ h\in \mathcal{D}(M)\,|\, f\circ h^{-1} = f \}.
\end{align*}
There is a canonical inclusion $j:\mathcal{S}_{R}(f)\hookrightarrow \mathcal{S}_{LR}(f)$ given by $j(h) = (\id_{\mathrm{Im}(f)}, h)$.
Denote by $\mathcal{S}_{R}^{\id}(f)$ a connected component of $\mathcal{S}_R(f)$ containing $\id_M.$

Homotopy properties of these spaces and their connected components are well studied for some ``wide'' class smooth functions with isolated singularities on surfaces by S.~Maksymenko, see \cite{Maksymenko:AGAG:2006, Maksymenko:UMZ:ENG:2012} and also a review paper \cite{Maksymenko:2021:review}.
In \cite{Feshchenko:Arxiv2023}  we studied such questions on  stabilizers (see Theorem \ref{thm:stab-homo-prop} below)  for a broader class of smooth functions on surfaces with non-isolated singularities $\F$, whose connected components of the set of critical points of $f$ are either  isolated critical points, or critical circles:

\begin{definition}[cf. Definition 1.1  \cite{Feshchenko:Arxiv2023}]\label{def:class-F}
{\rm
Let $M$ be compact surface possible with boundary $\partial M$, and
 $\F$ be a class of smooth functions on $M$ which satisfy:
\begin{enumerate}
	\item  a function $f|_{\partial M}$ is locally constant,
	\item a set of critical points $\Sigma_f$ of $f$ is a disjoint union of smooth submanifolds of $M$ and $\Sigma_f\subset \mathrm{Int}(M),$
	\item  for each connected component $C$ of $\Sigma_f$ and each critical point $p\in C$
	there exist a local chart $(U,\phi:U\to \bbR^2)$  near $p$  with $\phi(p) = 0$ and a chart $(V, \psi:V\to \bbR)$ near $f(p)\in P$ such that $f(U)\subset V$ and a local representation $f_p=\psi\circ f\circ \phi^{-1}:\phi(U)\to \psi(V)$ of $f$ is
	\begin{enumerate}
		\item[(3.a)] either a polynomial homogeneous polynomial $f_p$ without multiple factors,
		\item[(3.b)] or  is given by $f_p(x,y) = \pm y^{n_C}$ for some $n_C\in \bbN_{\geq 2}$ depending of $C.$  
	\end{enumerate}
\end{enumerate}
}
\end{definition}

It is known that for $f\in \F$ the stabilizer $\mathcal{S}_{R}(f)$ is strong deformation retract of $\mathcal{S}_{LR}(f)$, so   the inclusion $j:\mathcal{S}_{R}(f)\hookrightarrow \mathcal{S}_{LR}(f)$ as above is a homotopy equivalence, see \cite[Theorem 1.3]{Maksymenko:BulletinSciMath:2006}.
Note that recently we established the homotopy type of the connected component $\mathcal{S}^{\id}_{R}(f)$ of the identity map in $\mathcal{S}_{R}(f)$ for functions from $\F$, see Theorem \ref{thm:stab-homo-prop} below.

  \subsection{Functions from $\FF$ and their normal forms} 
Consider the following class subclass $\FF$ of $\F$:
\begin{definition}\label{def:class-FF}
	{\rm 
		Assume that $M$ is an {\it oriented} surface. Denote by $\FF$ a class of functions
		which satisfy (1), (2), (3.b) of Definition \ref{def:class-F}, and $(3.a')$ instead of (3.a):
		\begin{enumerate}[leftmargin=1in]
			\item[($3.a'$)] either a polynomial $f_p:\bbR^2\to \bbR$, given by $f_p(x,y) = \pm ( x^2+ y^2)$.
		\end{enumerate}
	}
\end{definition}
 Generalities of function from $\FF$ will be discussed in Section \ref{sec:hom-prop}, but here we only mention that the class $\FF$ is ``small'' in the sense that functions from $\FF$ appear on several ``simple'' surfaces:
\begin{proposition}\label{prop:properties-f}
	Let $f$ be a function from $\FF$. 
	
	{\rm (1)} Then  $M$ is diffeomorphic to  one of the following four surfaces: a cylinder $S^1\times [0,1]$, a disk $D^2,$ a sphere $S^2$ or a torus $T^2$.
	
	{\rm (2)} If $M$ is diffeomorphic to 
	\begin{itemize}
		\item  a cylinder or a torus, then $f$ has no isolated critical points,
		\item a disk, then $f$ has one Morse extremum (minimum or maximum),
		\item a sphere  then $f$ has two Morse extrema (minimum and/or maximum). Moreover, all four cases are possible (depends on the number of critical circles).
	\end{itemize}
\end{proposition}
\noindent Notice that surfaces from Proposition \ref{prop:properties-f} admit smooth semi-free actions of a circle $S^1$.

Our motivation to study functions from $\FF$ is the following result on 
 the homotopy type of $\mathcal{S}_R^{\id}(f)$ for functions from $\F$ mentioned above.
\begin{theorem}[Theorem 1.2 \cite{Feshchenko:Arxiv2023}, cf. Theorem 1.3 \cite{Maksymenko:AGAG:2006}]\label{thm:stab-homo-prop} Let $f$ be a function from $\F$. Then $\mathcal{S}_R^{\id}(f)$ is either contractible or homotopy equivalent to $S^1$. In particular
	$$
	\mathcal{S}_R^{\id}(f) \simeq \begin{cases}
		S^1, &\text{if $M$ is oriented and $f\in \FF$},\\
		\mathrm{pt}, & \text{otherwise}.
	\end{cases}
	$$	
\end{theorem}
Our aim is to give an ``analytic'' description of all functions from this class, see Theorem \ref{thm:thm-main-2}.
There are some trivial examples of functions from $\FF$ that are easy to write by hand:
\begin{definition}\label{example:height}
	\normalfont
		Let $f_0:M_0\to P$ be  a smooth function  from $\mathcal{F}^{\circ}$
	\begin{enumerate}
		\item[($1_0$)] $M_0 = \CYL = \{(z,s)\,|\, z\in \bbC, |z| = 1,\, 0\leq s\leq 1\}$ is a unit cylinder,  and 
		$f_0:S^1\times [0,1]\to \bbR$ is given by $f_0(z, s) = s$,
		\item[($2_0$)] $M_0 = \DIS = \{(x,y)\in\bbR^2\,|\, x^2+y^2\leq 1\}$ is a unit $2$-disk,
		and
		$f_0:D^2\to \bbR$ is given by $f_0(x,y) = x^2 + y^2$,
		\item[($3_0$)] $M_0 = \SPH = \{(x,y,z)\in \bbR^3\,|\, x^2+y^2+z^2=1\}$ is a unit sphere, and $f_0(x,y,z):S^2\to \bbR$ is given by  $f_0(x,y,z) = z$,

		\item[($4_0$)] $M_0 = \TOR = \{ (w, z)\in\bbC^2\,|\,|z| = |w| = 1   \}$ is a unit $2$-torus, and $f_0:T^2\to S^1$ is given by $f_0(w,z) = z$.
	\end{enumerate}
Note that these functions do not have critical circles. We will call them {\bf prime functions}. 
\end{definition}
\begin{theorem}[Main theorem]\label{thm:thm-main-2}
 A function $f\in \FF$ admits the following decomposition 
\begin{equation}\label{eq:f-funct}
f = \varkappa\circ \fz\circ h^{-1}
\end{equation}
where $f_0\in \mathcal{F}^{\circ}(M_0, P)$ is a prime function, 
for some diffeomorphism $h:M_0\to M$, and  a smooth function $\varkappa:f_0(M_0)\to P$ which  satisfies the following conditions:
\begin{enumerate}
	\item[{\rm(A)}] $\varkappa$ has the only finite number of critical points in which it is not flat, i.e., not all derivatives of $\varkappa$ at each critical point vanish,
	\item[{\rm (B)}] $\varkappa$ does not have critical points at $f_0(\Sigma_{f_0})$ and $f_0(\partial M)$,
\end{enumerate} 
where $\Sigma_{f_0}$ is the set of critical points of $f_0$.
A factorization \eqref{eq:f-funct} is not unique and depends on the choice of  $h$.  In particular, if $f$ has no critical circles, then $\varkappa$ is a diffeomorphism.
\end{theorem}
We will call factorization \eqref{eq:f-funct} a {\it normal form of $f$.}
The fact that  if $f\in \FF$ has no critical circles, then $\varkappa_{f}$ is a diffeomorphism, so $f$ is smoothly equivalent to $f_0$ is known,  see Lemma \ref{lm:smooth-equiv}.

\subsection*{Discussion of the result}
Due to a famous result of  H.~Whitney \cite{Whitney:DukeMathJ:1943}, for any {\it even} $C^{\infty}$-function $f:[-\varepsilon, \varepsilon]\to \bbR$, $\varepsilon>0$ there exists a unique $C^{\infty}$-function  $\alpha:[0, \varepsilon^2]\to \bbR$ such that $f$ can be presented as $f(x) = \alpha(x^2)$, $x\in  [-\varepsilon, \varepsilon]$. Obviously that $\alpha$ is defined by $\alpha(t) = f(\sqrt{t})$ and $\alpha$ is smooth on $(0,\varepsilon^2].$  Whitney showed that  $\alpha$ is smooth at $0.$
A function $f_0:[-\varepsilon, \varepsilon]\to \bbR$ given by $f_0(x) = x^2$ is the ``simplest'' even function and  since $\alpha$ is unique defined, the decomposition $f = \alpha\circ f_0$ can be called a ``normal form'' of an even function $f$ (in spirit of e.g., a Jordan normal form in linear algebra, or a  Poincar\'e normal form in the theory of dynamical systems, Whitney's normal forms of fold and pleat singularities, etc). Note also that $f$ is obtained  by smoothly changing values of $f_0$ by the function $\alpha$.

This similarity can be observed in Theorem \ref{thm:thm-main-2}, i.e.,
for each $f\in \FF$ there exist $h$ and $\varkappa$  as above, and such that a function $f\circ h:M\to P$ is obtained from a prime function $f_0:M_0\to P$ by  smoothly changing its values using a function  $\varkappa$ on the image of $f_0.$ So $f$ is  obtained as \eqref{eq:f-funct} from the ``simplest'' one. A function $\varkappa$ plays the role of a  ``generator'' of critical circles for $f$.

\subsection{Structure of the paper}
The paper is organized in 10 sections.

We begin the paper with a review of the general properties of functions from the class 
$\FF$ (Section \ref{sec:hom-prop})
paying particular attention to the behavior of such functions near critical circles. Next we investigate the restrictions on the number of extremal circles (Lemma \ref{lm:torus-crit-circles}).

The next  Section \ref{sec:Hfilds-normaliz} is devoted to the Hamiltonian fields of functions from the class  $\FF$, and 
a ``new'' concept of an $H$-field (Proposition \ref{prop-main}) for such functions.

In Section \ref{sec:free-semifree}, we provide the necessary general background on smooth free actions of the circle on itself, as well as free actions of the circle on surfaces. The main focus is on the conjugacy of free actions.

Semi-free smooth actions of the circle on surfaces are studied in Section \ref{subsec:normalized-vf}. Here, we also address the question of how semi-free circle actions on surfaces are related to the flows of 
$H$ -fields (as in Section \ref{sec:Hfilds-normaliz}) of functions from the class $\FF$.
The main result of this section is Lemma \ref{lm:h-conj-id}, which concerns the simplification of a diffeomorphism that conjugates two free and smooth $S^1$-actions on a cylinder
under the assumption that these actions coincide on certain subsets.

In Section \ref{sec:combinatirics}, we prove Proposition \ref{prop:properties-f} and Lemma \ref{lm:torus-crit-circles}. We also define an order relation on the set of critical circles of a given function needed for the proof of Proposition \ref{prop-main} in Section \ref{sec:proof-prop-main}.

Section \ref{sec:prof-main-f} contains the proof of Theorem \ref{thm:thm-main-2} for functions on a cylinder and a torus.

In Section \ref{sec:Axially}, we prove two technical results that are necessary for the proof of the Theorem \ref{thm:thm-main-2} for functions on a disk and a $2$-sphere   in Section \ref{sec:prof-main-f-2}.

\subsection{Acknowledgments}The author would like to express deep gratitude to  Sergiy Maksymenko and Yevgen Polulyakh for insightful discussions.


\section{General discussions of the class $\FF$}\label{sec:hom-prop}

Let $f$ be a function from $\FF$. Denote by $\CritP$ a set of isolated critical points of $f$ and by $\CritC$ a set of critical circles of $f$. 
A function $f$ defines a foliation with singularities $\Delta_f$ on $M$: a leaf $\gamma$ of $\Delta_f$ is either an isolated critical point, or a connected component of $f^{-1}(c)\setminus \CritP$ for some $c\in P.$
A subset $U\subset M$ consisting of leaves of $\Delta_f$ will be called {\it $f$-foliated subset}.

By ($3a'$) Definition \ref{def:class-FF} an isolated critical point $p$ of $f$ is a local extremum being Morse minimum or maximum. 
Proposition \ref{prop:properties-f} (will be proved in Section \ref{sec:combinatirics})  imposes restrictions on the number of such points. Also note that if $f$ has no critical circles, then $f$ is a Morse function and Proposition \ref{prop:properties-f} is a simple consequence of Morse equalities $\chi(M)= |\CritP|\geq 0$.

The rest of the section is devoted to studying the behavior of the functions from $\FF$ in the neighborhood of critical circles.
Let $C$ be a critical circle of $f$. 
 By the definition of the class $\mathcal{F}^{\circ}$ for any $p\in C$ a function $f$ has a local representation $f_p(x,y) = \pm y^{n_C}$ on some chart $U$ near $p$ for some $n_C\geq 2$. A critical circle $C$ of $f$ will be called {\it extremal circle of $f$} if $n_C$ is even, and {\it non-extremal} otherwise.
An extremal circle is called {\it maximal} if $f$ has a local maximum at $C$, and {\it minimal} otherwise. 

Note that each point $p\in C$ is singularity of $\mathrm{grad}f$ but
a trajectory of $\mathrm{grad}f$ has different behavior near extremal and non-extremal critical circle $C$.
Let $\mathbf{G}_t:M\to M$ be a flow of $\mathrm{grad}f$ and $N$ be a regular neighborhood of $C$. If $C$ is extremal ($n_C$ is even), then either  $\lim\limits_{t\to \infty}\mathbf{G}_t(z)\in C$ ($C$ is maximal), or  $\lim\limits_{t\to - \infty}\mathbf{G}_t(z)\in C$ ($C$ is minimal) for $z\in N$, i.e.,
 trajectories of $\mathrm{grad}f$  on $N$ converge to $C$ or diverge to $C$.
If $C$ is not extremal ($n_C$ is odd), then trajectories of $\mathrm{grad}(f)$ does not converge to $C$.

In general a function $f$ from $\FF$ does not have any restrictions on the number or critical circles (for extremal as well as for non-extremal) if $M$ is not diffeomorphic to $T^2$.
If $M\cong T^2$, then $f$ can have any number of non-extremal critical circles, but there is a restriction on the number of extremal circles:
\begin{lemma}\label{lm:torus-crit-circles}
	If $M$ is diffeomorphic to $T^2$, then $f$ has  an even number of extremal critical circles.
	In particular, if $f$ is null-homotopic, then  it has at least two extremal critical circles.
\end{lemma}
\noindent We will prove this lemma in \textsection\ref{subsec:proof-lema-torus}.

\section{$H$-like fields for functions from $\FF$}\label{sec:Hfilds-normaliz}
\subsection{Hamiltonian vector fields}\label{ssec:Ham-skew}
Let $M$ be a smooth and {\it oriented} surface.
It is well-known that $M$ admits a symplectic form $\omega:TM\times TM\to \bbR$, i.e., a non-degenerate skew-symmetric $2$-form. From the fact that $\omega$ is non-degenerate follows that the map $\omega^{\flat}:TM\to T^*M$ given by $\omega^{\flat}_p(X) = \omega_p(X,\cdot)$ for each $p\in M$ is a bundle isomorphism.  Let $f$ be a smooth $P$-valued function on $M$ and $df:TM\to TP$ be its differential.
Since $P$ is either $\bbR$ or $S^1$ it is well-known that  $TP$ is a globally trivializable bundle; let $\zeta:TP\to P\times \bbR$ be its trivialization isomorphism. Therefore a  differential $df$ induces  a section $Df:TM\to \bbR$ of $T^*M$ defined as the composition 
$$
\xymatrix{
	Df:TM\ar[r]^-{df} & TP \ar[r]^-{\zeta}& P\times \bbR \ar[r]^-{p_2}& \bbR,
}
$$ 
where $p_2$ is the projection onto a second factor. 
A vector field $X_f = (\omega^{\flat})^{-1}(Df)$ on $M$ is called {\it a Hamiltonian vector field of $f$}. It is know that $X_f$ satisfies the following conditions:
\begin{itemize}
	\item singular points of $X_f$ correspond to critical points of $f$,
	\item $f$ is constant along trajectories of $X_f$, i.e., $X_f(f) = 0.$ In other words, $X_f$ is tangent to level-sets of $f$.
\end{itemize}

\begin{lemma}\cite[Lemma 5.1]{Maksymenko:AGAG:2006}\label{lm:H-like-f}
	Let $f:M\to P$ be a Morse function. Then there exists a vector field $F$ on $M$  such that the following holds true:
	\begin{enumerate}
		\item $f$ is constant along trajectories of $F$,
		\item a point $z$ is fixed for $F$ iff $z$ is an isolated critical point of $f$, i.e., $z\in \CritP$. Moreover for every singular point $z$ of $F$ a germ of $(F,z)$  is smoothly equivalent to a germ $(X_{f_z}, 0)$ of a Hamiltonian vector field $X_{f_z}$ of a local representation  of $f_z:\bbR^2\to \bbR$ given by 
		\begin{equation}\label{eq:Morse}
		f_z(x,y) = \pm x^2\pm y^2
		\end{equation}
		  with respect to a symplectic structure 
		  \begin{equation}\label{eq:Darboux}
		  \omega =\frac{1}{2} dx\wedge dy
		  \end{equation} 
		  on $\bbR^2$, i.e., $X_{f_z}$ has the form 
		$X_{f_z}(x,y) = \mp y\frac{\partial}{\partial x} \pm x\frac{\partial}{\partial y}.$
	\end{enumerate}
\end{lemma}
\noindent A vector field $F$ such in Lemma \ref{lm:H-like-f} is called a {Hamiltonian-like vector field for $f$.}
\begin{remark}
	{\rm Note that if $(U,\phi:U\to \bbR^2)$ and $(V,\psi:V\to \bbR^2)$ be charts such that a local representation $f_z = \psi\circ f\circ \phi^{-1}:\phi(U)\to \psi(V)$ near $z$ has the form \eqref{eq:Morse}, it is not necessary that $\omega$ has the form as \eqref{eq:Darboux} on $U$.  The converse is also true, namely if $\omega$ has the form as \eqref{eq:Darboux} (by Darboux's theorem), it is also not necessary that $f$ has the form \eqref{eq:Morse} on $U$. 
	Condition (2) of Lemma \ref{lm:H-like-f} requires that these two conditions simultaneously hold for $f_z$ and $\omega$.
}
\end{remark}

\subsection{$H$-like vector field for functions from $\F$}\label{subsec:H-like}
 Let $f$ be a function from $\F$ and $X_f$ be a Hamiltonian vector field of $f$ with respect to some symplectic form $\omega$ on $M$.  Since each critical point of $f$ is a singularity of $X_f$, and in particular a point at a critical circle $C$ of $f$, it follows that a circle $C$ is a zero of $X_f$.

Let $\gamma$ be a trajectory of $X_f$ and $W$ be its regular connected neighborhood which consists of regular trajectories of $X_f$ being a cylinder. A symplectic form induces a canonical orientation on regular trajectories of $X_f$.
If $\gamma$ is regular, then each trajectory of $X_f$ on $W$ has the ``same'' orientation since $W$ is connected.

Assume that $\gamma=C$ is a critical circle of $f$.
We will say that trajectories of $X_f$ on $W$ have {\it the same orientation} on $W$, if a critical circle $C$ as a leaf of $\Delta_f$  can be canonically  oriented, and  they have {\it the opposite orientation on $W$} otherwise. 

It is easy to deduce that if $C$ is not an extremal circle ($n_C$ is odd), then $C$ can be canonically oriented, and if $C$ is an extremal circle, $C$ cannot be oriented ($n_C$ is even).
 This can be easily seen from an example.
\begin{example}
	{\rm 
		Let $S^1$ be a unit circle in $\bbC$ centered at $0$. Consider two functions $f,g:S^1\times [-1,1]\to \bbR$ from $\mathcal{F}^{\circ}(S^1\times [-1,1],\bbR)$ given by $f(z,t) = t^2$ and $g(z,t) = t^3$. 
		A curve $C = S^1\times\{0\}$
		is a	critical circle for $f$ and $g$, but it is extremal for $f$ and non-extremal for $g$. A differentials of $f$ and $g$ are given by $df = 2tdz$,  $dg = 3t^2dz$. So 
		with respect to a standard symplectic form $\omega = dz\wedge dt$ on $S^1\times [0,1]$ Hamiltonian vector fields for $f$ and $g$ have the form 
		$$
		X_f = -2t\frac{\partial}{\partial z},\qquad X_g = -3t^2\frac{\partial}{\partial z}.
		$$
		Let $p = (z_0,-1/2)$ and $q = (z_0, 1/2)$ be points in $S^1\times [-1,1]$. Then we have $X_f|_p = \frac{\partial}{\partial z}|_p$ and $X_f|_q = -\frac{\partial}{\partial z}|_q,$ while at each $z\in S^1\times [0,1]$ a  $(X_g)_z$ is always negative oriented on $T_z(S^1\times [0,1])$ with respect to a usual orientation of a cylinder. Therefore regular trajectories of $X_f$ ``change'' their orientations passing through $C$, while trajectories of  $X_g$ do not.\qed
	}
\end{example}

The following proposition is one of the key ingredients for the proof of Theorem \ref{thm:thm-main-2}. This result states that, roughly speaking, we can change  the vector field $X_f$ by ``removing'' its singular cycles. Such vector fields will be used to define a smooth semi-free $S^1$-actions on surfaces, see Section \ref{subsec:normalized-vf}.
\begin{proposition}\label{prop-main}
	Let $M$ be an oriented surface and $f$ be a function from $\FF.$ Then there exists a vector field $F$ with the only isolated singularities which satisfies conditions {\rm (1)--(2)} of {\rm Lemma \ref{lm:H-like-f}}.
\end{proposition}
To avoid confusions, a vector field $F$ for a function $f\in \FF$ such in Proposition \ref{prop-main} will be called an {\it $H$-field} for $f$. Proposition \ref{prop-main} will be proved in Section \ref{sec:proof-prop-main}.

\section{Free $S^1$-actions on surfaces}\label{sec:free-semifree}
We recall two definitions about group actions.
Let $G$ be a Lie group acting on a smooth manifold $M$ by two actions $\phi,\psi:M\times G\to M$. Actions $\phi$ and $\psi$ are called {\it conjugated} if there exists a diffeomorphism $h:M\to M$ such that  $\phi_g = h\circ \psi_g\circ h^{-1}$ for all $g\in G.$ A $G$-action $\psi$ is called {\it free}, it  $\psi(g,x) = x$ for some $x\in M$ implies $g = e_G$, where $e_G$ is an identity of $G$. 
The result of the action $\psi(x,g)$ of $g$ on $x$ we will also denote by $\psi_g(x)$ for  simplification of our notation.

\subsection{Free $S^1$-actions on itself}
We consider a circle $S^1$ as a smooth manifold with the underlying set  $\{z\in\bbC\,:\,|z| = 1 \}$. 
There are a natural smooth $\bbR$-action $\xi:S^1\times \bbR\to S^1$ on $S^1$ and an $S^1$-action on itself $\sst:S^1\times S^1\to S^1$ given by
\begin{equation}\label{eq:R-act-S^1}
	\xi(z,t) = ze^{2\pi i t},\qquad \sst(z,g) = zg.
\end{equation}
The action $\sst$ is a rotation of a circle and it will be called a standard $S^1$-action on itself. 
Actions $\xi$ and $\sst$ induce each other, i.e., the following formula holds:
\begin{equation}\label{eq:actions-link-S^1}
	\sst(z,e^{2\pi i t}) = \xi(z, t)
\end{equation}
for all $z\in S^1$ and $t\in \bbR.$
The next simple lemma is part of a mathematical ``folklore''.
\begin{lemma}\label{lm:diff-commutes-angles}
	Let $\psi:S^1\times S^1\to S^1$ be a free smooth $S^1$-action on itself.
	\begin{enumerate}
		\item Then $\psi$ is conjugated to $\sst$. So any two free smooth $S^1$-actions on itself are conjugated.
		\item If a diffeomorphism $h:S^1\to S^1$ commutes with $\psi$, i.e.,  $h\circ \psi_g = \psi_g\circ h$ for all $g\in S^1$, then $h = \psi_{g_0}$ for some $g_0\in S^1.$ 
	\end{enumerate}
\end{lemma}

\subsection{Free $S^1$-actions on $S^1\times [0,1]$ and $T^2$}
Denote by $B=[0,1] $ or $S^1$.
An  $\bbR$-action $\xi$ and $S^1$-action on $S^1$ given by \eqref{eq:R-act-S^1} induce a smooth $\bbR$-action $\xi:M\times\bbR\to M$ and $S^1$-action $\sst:M\times S^1\to M$, where $M = S^1\times B$ by the following rule:
\begin{equation}\label{eq:R-act-M}
	\xi(z,b,t) = (ze^{2\pi i t}, b), \qquad \sst(z,b,g) = (zg, b)
\end{equation}
For actions \eqref{eq:R-act-M} similarly to \eqref{eq:actions-link-S^1} the following formula holds
$$
\sst(z,b,e^{2\pi i t}) =\xi(z,b,t).
$$
The action $\sst$ will be also called a standard $S^1$-actions on $M = S^1\times B$.
The fact that we consider only free actions on the cylinder and the torus is a consequence of the following  result
\begin{lemma}[cf. Proposition 6.7, Theorems 6.10 and 6.22 \cite{Morita:TMM:2001}]\label{lm:free-S1-action}
Let $M$ be a smooth oriented compact surface and $\psi:M\times S^1\to M$ be a smooth and free $S^1$-action on $M$. Denote by $M/\psi$ an orbit space of the action $\psi$ on $M$, and 
let $p:M\to M/\psi$ be a canonical  projection. Then
\begin{enumerate}
	\item an orbit space $M/\psi$ is diffeomorphic to $B$,   $p:M\to M/\psi$ is isomorphic as a principle $S^1$-bundle to a trivial $S^1$-bundle $f_0:S^1 \times B\to B$, where $f_0(z,b) = b$ for all $(z,b)\in S^1\times B,$ i.e., there are diffeomorphisms $h:S^1\times B\to M$ and $\ell:B\to M/\psi.$ such that the diagram commutes:
	$$
	\xymatrix{
	S^1\times B \ar[d]_{f_0} \ar[r]^h_{\cong} & M \ar[d]^p\\
	B \ar[r]^{\ell}_{\cong} & 	M/\psi
}
	$$

	\item an action $\psi$ on $M$ is conjugated to $\sst$ on $S^1\times B$ by a diffeomorphism $h$ as in {\rm (1)}, i.e., $\psi_g = h\circ \sst_g\circ h^{-1}$. Consequently, any  two smooth free  $S^1$-actions on $M$ are conjugated.
\end{enumerate}	
\end{lemma}
Other results on circle bundles on manifolds the reader can find in \textsection 6.2 \cite{Morita:TMM:2001}.

\section{Semi-free $S^1$-actions on surfaces}\label{subsec:normalized-vf}
Let $\psi:M\times G \to M$ be a smooth action of a Lie group $G$ on a smooth manifold $M$. 
 It is known that $\mathrm{Fix}(\psi)$ is a closed subset of $M.$
An action $\psi:M\times G \to M$ is called {\it semi-free}, if it is free on the complement of its fixed points $\mathrm{Fix}(\psi)$.

If $M$ is a smooth oriented and compact surface and $\psi$ is a semi-free $S^1$-action on $M$ one can establish, by using for example  slice theorem \cite[p. 139]{Koszul:GD:1953} that $\mathrm{Fix}(\psi)$ is a  finite subset of $M$  and  $\chi(M) = |\mathrm{Fix}(\psi)|\geq 0$, where $\chi(M)$ is an Eulerian characteristic of $M$.
The last formula is also a simple consequence of Atiyah–Bott–Berline–Vergne integration formula, see \cite[Theorem 30.2]{Tu:EquivariantCohomology:2020}.
Therefore $M$ is diffeomorphic to  one of the following list $S^1\times [0,1]$,  $T^2$,  $D^2$, $S^2.$

It is also known that any smooth $S^1$-action on a surface $M$ is uniquely determined by some vector field on $M$, since $S^1$ is connected Lie group.
In this paragraph we will show how such $S^1$-actions naturally ``occur'' from  smooth functions on surfaces.

Let $f$ be a function from $\mathcal{F}^0(M,P)$, $F$ be an $H$-field of $f$ with the flow $\bfF_t:M\to M$, $t\in \bbR$, see Proposition \ref{prop-main}.
A function $f$ does not have saddle critical points, so each trajectory of $F$ is periodic.
There exist a function $\per_F:M\setminus\CritP\to \bbR$ associating to each non-singular point $x\in M\setminus\CritP$ its prime period $\per_F(x)$  with respect to the flow $\bfF$, see \cite[\textsection 3]{Maks:reparam-sh-map}.

Let also $\nu:M\to \bbR$ be a positive smooth function on $M$ which is constant along trajectories of $F$, i.e., $F\nu = 0.$
It is known that  a vector field $\nu F$ defines the same foliation on $M$ as $F$.
We can choose a function $\nu$ such that $\per_{\nu F}(x)=1$ for all $x\in M$, see \cite[Lemma 2.1, Lemma 3.1]{Maks:reparam-sh-map}. 
Near each isolated singularity  $z$ a vector field $\nu F$ has the form $(\nu F)_z = \mp y\frac{\partial}{\partial x} \pm x \frac{\partial}{\partial y}$. It is known that \cite[Theorem 1.4]{Maksymenko:Nel:2009}  a function $\per_F$ can be smoothly extended to a whole $M.$ 
Such reparamentized vector field will be denoted by the same latter $F$ and called a {\bf normalized $H$-field} for $f$.

The flow $\bfF$ of a normalized $H$-field $F$ satisfies  $\bfF(x,t) = x$ for each $x\in M$ i.e., $\bfF_1 = \id_{M}$.
 Therefore $\bfF$ induces  a smooth $S^1$-action   $\psi:M\times S^1\to M$ on $M$ given by
\begin{equation}\label{eq:ciecle-act}
	\psi(x, e^{2\pi i s})  = \bfF_{s}(x),
\end{equation}
see \cite[Corollary 3.3]{Maks:reparam-sh-map}.
So $S^1$ acts on $M$ by smooth ``shifts'' along orbits of the flow $\bfF.$ Since isolated singularities of $f$ corresponds to a singularities of $F,$ it follows that
the action \eqref{eq:ciecle-act} is free on $M\setminus\CritP$, which will be called  a {\bf semi-free $S^1$-action on $M$ induced by $F$ (or by $f$)}.

\begin{example}\label{ex:circle-act-primitive}
{\rm 
Note that standard $S^1$-actions $\sst$ on a cylinder $S^1\times [0,1]$ and a torus $T^2$ appear  as \eqref{eq:ciecle-act}. Denote by $B = [0,1]$ or $S^1$ so $M = S^1\times B$. Let $f_0:M\to B$ be a prime function. A function $f_0$ is a projection onto $B$ and does not have critical points. Let also $F_0$ be a normalized $H$-field of $f_0$. The flow $\bfF^0:S^1\times B\times \bbR\to S^1\times B$ is given by $\bfF^0_s(z, b) =(ze^{2\pi i s}, b)$. It is easy to see that the last formula is the definition of  $\sst(z,b, e^{2\pi i s})$.
}
\end{example}
It is known that  an orientation-preserving diffeomorphism $h$ of a cylinder $S^1\times [0,1]$ such that $h(S^1\times \{i\}) = S^1\times \{i\}$, $i = 0,1$ is isotopic to a such diffeomorphism $\tilde{h}:S^1\times [0,1]\to S^1\times [0,1]$ which is fixed on some neighborhood of the boundary $S^1\times [0,1]$. The following lemma  is a ``variation'' of this fact using smooth $S^1$-actions on a cylinder.
It is an important tool for the proof of  Theorem \ref{thm:thm-main-2}.

\begin{lemma}\label{lm:h-conj-id}
	Let $M$ be diffeomorphic to a cylinder $S^1\times [0,1]$, $\psi, \phi:M\times S^1\to M$ be two smooth $S^1$-actions on $M$  such that $\phi$ is given by \eqref{eq:ciecle-act} for some flow $\bfF_t:M\to M$. Let also $h:M\to M$ be a diffeomorphism which conjugates $\psi$ and $\phi,$ i.e., $\psi_g = h\circ \phi_g\circ h^{-1}$ {\rm (}such $h$ always exists due to {\rm Lemma \ref{lm:free-S1-action})}.
	Let also $U$ and $V$ be open subspaces {\rm(}not necessary connected{\rm)} of $M$ consisting of trajectories of $\bfF$ and such that $\overline{V}\subset U$. If $\psi = \phi$ on $U$ then $h$ can be isotoped to a diffeomorphism $\tilde{h}:M\to M$ such that $\psi_g = \tilde{h}\circ \phi_g\circ \tilde{h}^{-1}$ and 
	\begin{equation}\label{eq:new-diffeo-conj}
		\tilde{h} = \begin{cases}
			\id, & \text{on } \overline{V},\\
			h, & \text{on } M\setminus U.
		\end{cases}
	\end{equation}
\end{lemma}
\begin{proof}
Since $\psi$ coincides with $\phi$ on $U$, it follows that $\phi_g =h\circ \phi_g\circ h^{-1}$ on $U$, and so $\phi_g\circ h = h\circ \phi_g$ on $U$ for each $g\in S^1.$ Let $\gamma\subset U$ be a trajectory of $\phi$. Then by Lemma \ref{lm:diff-commutes-angles} there exists $g_0\in S^1$ such that $h = \phi_{g_0}$ on $\gamma.$ Then for $t_0\in \bbR$ such that $g_0 = e^{2\pi i t_0}$ a diffeomorphism $h$ on $\gamma$ has the form $h = \bfF_{t_0}$. A number $t_0$ is smoothly depending on $\gamma\subset U$, so on $U$ a diffeomorphism $h$ became
$h = \bfF_{t_0}$, where $t_0:U\to \bbR$ is a smooth function on $U$ which satisfies $t_0|_{\gamma}\in \bbR$ for each $\gamma\in\Delta_{\bfF}$, i.e., $Ft_0 = 0$. Such diffeomorphism $h$ is called a ``shift'' along trajectories of $\bfF$.
We extend a function $t_0$ to  a smooth function on a whole $M$ denoted by the same symbol $t_0:M\to \bbR$.

Consider a smooth bump function  $\delta:M\to \bbR$ which satisfies the following conditions:
\begin{enumerate}
	\item $\delta$ is constant on trajectories of $\bfF$, i.e., $F\delta = 0$,
	\item $\delta = 1$ on $\overline{V}$ and $\delta = 0$ on $M\setminus U$.
\end{enumerate}
 A map  $\bfF_{\delta t_0}:M\to M$ is a diffeomorphism due to  \cite[Theorem 19]{Maksymenko:TA:2003}.
Note that  $\bfF_{\delta t_0}$ coincides with $h$ on $\overline{V}$ and is identity on $M\setminus U$.  

Define a diffeomorphism $\tilde{h}:M\to M$ given by $\tilde{h} = \bfF_{\delta t_0}^{-1}\circ h$. 
A diffeomorphism $\tilde{h}$ is isotopic to $h$ since the isotopy $H_s:M\to M$, $s\in [0,1]$ between $h$ and $\tilde{h}$ can be defined by the formula 
\begin{equation}\label{eq:H-isotopy}
H_s = \bfF_{s\delta t_0}^{-1}\circ h. 
\end{equation}
Indeed, $H_0 = h$ and $H_1 = \tilde{h}$, and each $H_s$ are diffeomorphisms.
It is easy to see that $\tilde{h}$ satisfies \eqref{eq:new-diffeo-conj}, so it conjugates $\psi$ and $\phi$ on $\overline{V}$ and $M\setminus U$. 

To prove that $\tilde{h}$ conjugates these actions on $\overline{U\setminus V}$ we claim that $\bfF_{\delta t_0}$ commutes with $\psi_g$. Indeed, a function $\delta t_0$ is constant along each orbit $\gamma$ of $\bfF$ on $U$, so for each $x\in \gamma$ we have $\delta(x)t_0(x) = c\in \bbR$. Therefore $\bfF_{\delta _0}|_{\gamma}(x) = \phi_c(x)$ for $x\in \gamma,$ and since $\phi$ coincides with $\psi$ on $U$ we obtain $\bfF_{\delta t_0}|_{\gamma}(x) = \psi_c(x)$. An action $\psi$ obviously commutes with a rotation $\psi_c$.

The fact that $\tilde{h}$ conjugates $\phi$ and $\psi$ on $\overline{U\setminus V}$ can be checked as follows:
\begin{align*}
	\tilde{h}\circ \phi_g\circ	\tilde{h}^{-1} &=  \bfF_{\delta t_0}^{-1}\circ h\circ \phi_g\circ h^{-1}\circ \bfF_{\delta t_0}\\
	&= \bfF_{\delta t_0}^{-1}\circ \psi_g\circ \bfF_{\delta t_0} \\
	&=\bfF_{\delta t_0}^{-1}\circ\bfF_{\delta t_0}\circ \psi_g & \text{($\bfF_{\delta t_0}$ commutes with $\phi_g$ on $U$)}\\
	&=\psi_g.
\end{align*}
The proof is completed.
\end{proof}

\section{Combinatorial structure of functions from $\FF$}\label{sec:combinatirics}
Let $f$ be a function from $\FF.$ By Definition \ref{def:class-FF} a function $f$ has finitely many critical points and critical circles.

\subsection{Proof of Proposition \ref{prop:properties-f}}\label{sec:Prof-lm-prop}
 If $f$ has no critical circles then $f$ is a Morse function without saddles. In this case Proposition \ref{prop:properties-f} follows from Morse equalities:
 \begin{equation}\label{eq:Morse-eq-crit}
 	\chi(M) = |\Sigma^p_g|\geq 0.
 \end{equation}

 We assume that $\CritC =\{C_1, C_2,\ldots, C_n\}$ for some $n\geq 1$. 
  Proposition \ref{prop:properties-f} is the consequence of the following 
 \begin{lemma}\label{lm:decomp-cyl}
 	Let $f$ be a function from $\FF$ with  $\CritC = \{C_1,C_2,\ldots, C_n\}$. Denote by $\mathcal{N} = \{N_1, N_2, \ldots, N_s\}$ a set of oriented and connected surfaces with boundary such that
 		\begin{equation}\label{lm:decomposition}
 		\overline{M\setminus\bigcup_{i = 1}^n C_i}= \bigcup_{j = 1}^s N_j.
 		\end{equation}
 	 Then $N\in\mathcal{N}$ is diffeomorphic to either $S^1\times [0,1]$ or to $D^2.$
 \end{lemma}
 \begin{proof}
 	Let $N$ be a surface from $\mathcal{N}$,
 	and $W_i$ be an open connected $f$-foliated neighborhood of $C_i\in \CritC$ such that $W_i\setminus C_i$ contains no critical points of $f.$ We put $W = \bigcup_{i = 1}^n W_i$ and $Q = \overline{N\setminus W}.$ Note that $Q$  does not contain critical circles of $f$. Then a restriction $f|_Q:Q\to P$ is Morse function without saddles. 
 	 From Morse equalities \eqref{eq:Morse-eq-crit} we obtain $\chi(Q) = |\Sigma_{f|_Q}|\geq 0 $.
 	Since $Q$ has the boundary, it follows that $Q$  is diffeomorphic to either a disk $D^2$ or a cylinder $S^1\times [0,1]$. By construction, $W_i$ is a cylinder containing only one critical circle of $f$ belonging to $N$. Therefore $\chi(N) = \chi(Q)\geq 0$ and so $N$ is diffeomorphic to either a cylinder $S^1\times [0,1]$ or to a disk $D^2$, since $N$ has the boundary.
 \end{proof}	
	(1) By Lemma \ref{lm:decomp-cyl}, a surface $M$ is presented as some cylinders  and disks bounded by critical circles of $f$ and attached together along critical circles of $f$.
 	Since $\chi(C_i) = 0$ for each $i = 1,2,\ldots, n$, it follows that  
 	\begin{equation}\label{eq:euler-char-sum}
 	\chi(M) = \sum_{j} \chi(N_j) \geq 0
	\end{equation}
 	and so $M$ is diffeomorphic to one of the following list: $S^1\times [0,1]$, $T^2$, $D^2$, $S^2$.
 	
 	Statement (2) of Proposition  \ref{prop:properties-f} follows from \eqref{eq:euler-char-sum} and Morse equalities \eqref{eq:Morse-eq-crit} applying to $f|_{Q}$ form $Q\subset N\in \mathcal{N}$, where $Q$ is as in Lemma \ref{lm:decomp-cyl}. The proof of Proposition  \ref{prop:properties-f} is completed.
 	\qed
 	
 	
 	The following corollary is a simple consequence of Proposition \ref{prop:properties-f} and the behavior of the gradient vector field of $f$ near critical circles, see Section \ref{sec:hom-prop}.

\begin{corollary}\label{cor:cor3}
	Let $L$ be a closed connected subcylinder of $M$ bounded by extremal circles $C$ and $C'$, and such that $\mathrm{Int}L$ does not contain any other extremal circles of $f$. The following holds true: if $C$ is maximal (minimal), then $C'$ is minimal (maximal).
\end{corollary}

\subsection{Recapture of a surface structure}\label{subsec:orering} 
Let $f$ be a function from $\FF$ and $\CritC = \{C_1,C_2,\ldots, C_n\}$ be its set of critical circles. Let also $\mathcal{N}$ be the set of subsurfaces of $M$ such in Lemma \ref{lm:decomp-cyl}.
So a surface $M$ is a union of surfaces $N_j$ being cylinders or disks, whose boundaries are critical circles of $f$. It is possible to endow some order relation on the set of critical circles $\CritC$ of $f$. We consider two cases.

{\it Case 1. $M$ is not diffeomorphic to $T^2$}.
Each $C\in \CritC$ separates $M$ and all of curves from $\CritC$ are mutually disjoint and pairwise isotopic to each other.   	In particular we have $|\mathcal{N}| = |\CritC|+1$.

We can rearrange curves from $\CritC$ with respect to the following order relation:
$C_i\preceq C_j$ iff $M_i\subset M_j$, where $M_i$ is {\it a slice of $M$ under $C_i$}, i.e., a connected component of $\overline{M\setminus C_i}$ containing
\begin{itemize}
	\item a fixed connected component of $\partial M$ if $M\cong S^1\times [0,1]$,
	\item an isolated critical point $z$ of $f$, if $M \cong D^2$,
	\item a fixed isolated critical point $z$ of $f$ where $f$, if $M \cong S^2$.
\end{itemize}

This order relation induces some relation on the set $\mathcal{N}$ from Lemma \ref{lm:decomp-cyl}. So we reenumerate the set $\mathcal{N}$ of subsurfaces $N_j$ from \eqref{lm:decomposition}   such that $N_1 = M_{1}$ and $N_i$ be such element of $\mathcal{N}$ such that $N_{j+1}  = \overline{ M_{{j+1}}\setminus M_{j}}$. 

If $M$ is diffeomorphic to $D^2$, then $N_{n+1}$ is a cylinder which contains $\partial D^2$, and if $M$ is diffeomorphic to $S^2,$ a surface $N_{n+1}$ is diffeomorphic to a disk which contains the other extreme, which is not fixed above.

{\it Case 2. $M$ is diffeomorphic to $T^2$.}
If $M = T^2$, then each curve $C\in \CritC$ does not separate $T^2$, any two $C$ and $C'$ of them are mutually disjoint and isotopic to each other. In particular we have $|\mathcal{N}| = |\CritC|.$
So we can assume that a curves from $\CritC$ is cyclically enumerated $\CritC = \{C_1,C_2,\ldots, C_n, C_{n+1} = C_1\}$. 
 Denote by $N_i$  a cylinder bounded by critical circles $C_i$ and $C_{i+1}$.
We also set:
$$
C_i := C_{i\;\mathrm{mod} n +1},\qquad N_i := N_{i\;\mathrm{mod} n+1}.
$$
So $T^2$ is obtained by attaching $n$ cylinders $N_i$ to each other by a cycle.

\subsection{Proof of Lemma  \ref{lm:torus-crit-circles}}\label{subsec:proof-lema-torus}
Let $M$ is diffeomorphic to $T^2$, $f\in \FF$, and $\CritC$ be the set of critical circles of $f$, with $|\CritC| = n\geq 0.$
Denote by $E = \{A_1,A_2,\ldots, A_k\}\subset \CritC$  the set of extremal circles of $f$ numbered with respect to a cyclic order on $\CritC$ for $k\leq n$, and  by 
 $L_i$ a subcylinder in $M$ bounded by $A_i$ and $A_{i+1}$, where all indexes are taken modulo $k$. Assume the converse it true, i.e., $|E|$ is odd, i.e., $k = 2r-1$ for some $r\geq 1$. Then by Corollary \ref{cor:cor3}. If $A_1$ is maximal, then $A_{2i-1}$ is maximal and $A_{2i}$ is minimal for each $i = 1,2\ldots, k=2r-1$. But $A_1 =A_{2r}$, which leads to a contradiction.

If $f$ is null-homotopic, then we regard $f$ as real-valued function. Since $f$ has no isolated critical points and $T^2$ is compact, then $f$ attains its minimum and maximum values. Therefore $f$ has at least two critical circles.

\section{Proof of Proposition \ref{prop-main}}\label{sec:proof-prop-main}
Let $f$ be a function from $\FF$ with the set of critical circles $\CritC$. 
Let $X_f$ be a Hamiltonian vector field of $f$ with respect to some symplectic structure $\omega$ on $M$. 
 
\subsection{Case $M = S^1\times[0,1]$}\label{subsec:H-cylinder}
A function $f:M\to \bbR$ has no isolated critical points, see Proposition \ref{prop:properties-f}.
If $f$ has no critical circles, then $f$ is a Morse function and a Hamiltonian vector field $X_f$ obviously satisfies condition of Proposition \ref{prop-main}.

Assume that $\CritC = \{C_1,C_2\ldots, C_n\}$ ordered as in \textsection\ref{subsec:orering} for some $n\geq 1$.
Let also $N_i$ be a cylinder bounded by $C_i$ and $C_{i+1}$, and $M_i$ is a slice under $C_i$, see \textsection\ref{subsec:orering} .
A Hamiltonian field $X_f$ automatically satisfies (1)  of Lemma \ref{lm:H-like-f}, but each critical circle $C_i\in \CritC$ is a zero of $X_f$. 
To obtain a needed vector field $F$ we have to do 3 steps: we need to

 (1)   define a new vector field $Y$ on $M$ such that any two regular trajectory $\gamma_1$ and $\gamma_2$ of $Y$ on $M$ have the ``same orientation'' in the following sense: $\gamma_1$ and $\gamma_2$ has the same orientation if any trajectory of $Y$ on a cylinder $Q$ bounded by $\gamma_1$ and $\gamma_2$ can be canonically oriented;
 
 (2) consider a vector field $Z$ on $M$ such that $Z$ is tangent to each level-sets of $f$ on some neighborhood of a critical circle $C\in \CritC$, $Z$ is has no zeros in it, and $Z$ is zero outside some ``bigger'' neighborhood of $C$. 
 
(3) Finally using the partition of the unity attach these two vector fields to obtain a vector field $F$, which has no zeros in $M$, so it satisfies Proposition \ref{prop-main}.

{\it Step 1.}
Let $V_i$ and $W_i$ be open connected and $f$-foliated neighborhoods of a critical circle $C_i$ of $f$ such that $\overline{V}_i\subset W_i$,  and $W_i$ does not contain other critical points of $f$. We set $V = \bigcup_{i=1}^n V_i$ and $W = \bigcup_{i = 1}^n W_i$.
 Let also $\alpha:M\to \bbR$ be a smooth bump function such that $X_f \alpha=0$ and 
$$
 \alpha = \begin{cases}
	0, &\text{on } \overline{V},\\
	1, & \text{on } \overline{M\setminus W}.
\end{cases}
$$ 
Denote by $X$ a vector field on $M$ given by $X = \alpha X_f$. A vector field $X$ has a zero locus in $\overline{V}$.

Define a vector field $Y$ on $M$ by the following induction procedure:
set a vector field $Y_1$ on $N_1 = M_{1}$ by the formula $Y_0 = X|_{N_1}$, and define $Y_i$ as follows:
$$
Y_{i+1} = \begin{cases}
	Y_{i}, & \text{on } M_{i},\\
	a_{i+1} X|_{N_{i+1}}, & \text{on } N_{i+1},
\end{cases}
$$
where $a_{i+1}$ is a constant defined by
$$
a_{i+1} = \begin{cases}
	+1, & \text{if $C_{i+1}$ is non-extremal critical circle},\\
	-1, & \text{otherwise}.
\end{cases}
$$
On the last step, when $i = n-1$, we obtain a vector field $Y = Y_{n}$.

{\it Step 2.} Since $\overline{W}_i$ is a cylinder, it follows that there exits a vector field $Z_i$ on $\overline{W}_i$ such that $Z_i$ is tangent to level-sets of $f|_{\overline{W}_i}$ and $Z_i$ is non-zero on $\overline{W}_i.$ Indeed such vector field $Z_i$ can be defined as follows. Let $p\in C_i$ be a critical point. There exists a chart $(U,\phi:U\to \bbR^2)$ near $z$ with $\phi(z) = 0$ and such that $f_z = f\circ \phi^{-1}:\phi(U)\to \bbR$ has the form $f_z(x,y) = \pm y^{n_{C_i}}$ for some $n_{C_i}\geq 2.$
Let $F_p$ be a vector field on $U$ given by $F_p = \frac{\partial}{\partial x}$. Assume that $W$ is covered $\overline{W} = \bigcup_{j=0}^n U_j$ by charts $(U_j, \phi_j:U_j\to \bbR^2)$ and $U_0 = U_p$. 
Then using a partition of the unity subordinating to this cover one can extend $F_p$ to a vector field on $\overline{W}_i$ denoted by $Z_i$. 

A vector field $Z_i$ can be chosen to be  co-directional with $Y$ on $\overline{W_i\setminus V_i}$, i.e., $Z_i = \lambda_i Y$ on $\overline{W_i\setminus V_i}$ for some non-negative function $\lambda_i:\overline{W_i\setminus V_i}\to \bbR$. Indeed, let $\gamma$ be a trajectory of $Y$ on  $\overline{W_i\setminus V_i}$. Since $T_p\gamma$ is an $1$-dimensional vector space, it follows that $(Z_i)_p = \lambda_i(p)Y_p$ on $\overline{W_i\setminus V_i}$, for some non-zero  $\lambda(p)\in \bbR$. A number $\lambda_i(p)$ is smoothly dependent on $p\in \overline{W_i\setminus V_i}$, so $Z_i = \lambda_iY$ for some non-zero smooth function $\lambda_i:\overline{W_i\setminus V_i}\to \bbR$. So by changing $Z_i$ to $-Z_i$ if needed, on can assume that $Z_i$ is codirectional with $Y$ on $\overline{W_i\setminus V_i}$.

Since connected components of $M\setminus W$ are cylinders it is easy to see that there exists a vector field $Z'$ on $M$ which extends $Z_i$ on $W_i$ for each $i = 1,2,\ldots,n$. Denote by $Z$ a vector field $(1-\alpha)Z'$. Vector field $Z$ has zero locus outside $W.$

{\it Step 3.} Finally we set $F = Y + Z$ on $M$.  A vector field $F$ is tangent to level-sets of $f$, has no zeros in $M$ since $Y$ and $Z$ are co-directional on $W\setminus V$ and zero locuses of $Y$ and $Z$ do not intersect. So $F$ is such as required in Proposition \ref{prop-main}.
The proof for the case of $M = S^1\times [0,1]$ is completed.

\subsection{Case $M = T^2$} A function $f:T^2\to P$ has no isolated critical points, see Proposition \ref{prop:properties-f}.
We reduce the proof in this case to the case of cylinder.

 Let $A$ be a regular and connected component of some level-set and $U$ be its open connected and $f$-foliated neighborhood, which does not contain critical circles of $f$. 
 Denote by $Q$ a cylinder $\overline{T^2\setminus A}$. A curve $A$ does not separate $T^2$, but it separates $U$, i.e., $\overline{U\setminus A}$ has two connected components, say, $U_0$ and $U_1$ being cylinders.

A function
$f|_Q\in\mathcal{F}^{\circ}(Q, P)$, so by \textsection\ref{subsec:H-cylinder} that there exist a vector field $F_Q$ without zeros on $Q$. Since $U$ contains only regular trajectories of $X_f$, it follows the definition of $F$ that 
\begin{equation}\label{eq:F-torus}
F_Q|_{U_i} = b_i X_f|_{U_i}
\end{equation}
for $i = 0,1$, where $b_i = \pm 1.$ By (3) of Proposition \ref{prop:properties-f} a function $f$ has even number of critical circles yielding $b_0 = b_1$ in \eqref{eq:F-torus}. So a vector field $F_Q$ on $Q$ can be extended on $A$ by setting $b_0X_f|_{A}$. The resulting vector field on $T^2$ will be also denoted by $F$. By definition $F$ has no zeros in $T^2$.

\subsection{Case $M =D^2$} A function $f:D^2\to \bbR$ has a unique isolated critical point $z$, see Proposition \ref{prop:properties-f}, and $n$ critical circles $\CritC = \{C_1, C_2, \ldots C_n\}$, $n\geq 0$. This critical point $z$ belongs to $N_1\in \mathcal{N}$ with respect to the order relation \textsection\ref{subsec:orering}.

Let $U, U'$ be an open neighborhood of $z$ consisting of connected components of level-sets of $f$ and such that $\overline{U}\subset \mathrm{Int}(N_1)$ and $\overline{U'}\subset U$. A restriction $f|_{\overline{U}}$ is a Morse function, so there exist a vector field $F_U$ on $\overline{U}$ such in Lemma \ref{lm:H-like-f}. Denote by $Q$ a cylinder $\overline{D^2\setminus U'}$. By \textsection\ref{subsec:H-cylinder} there exist a vector field called $F_Q$ on $Q$ which satisfies Proposition \ref{prop-main}.  We always  assume that $F_Q$ is codirectional with $F_U$ on $U'$, otherwise we redefine $F_Q$ as $-F_Q$. 
 Let $\delta:D^2\to \bbR$ be a smooth bump function with $X_f\delta = 0$ and $\delta =1$ on $\overline{U'}$ and $\delta = 0$ on $\overline{D^2\setminus U}$.
Extend vector fields $F_U$ and $F_Q$ to $D^2$ and define a vector field $F$  by the formula $\delta F_U+ (1-\delta)F_Q$. By its definition $F$ satisfies conditions of Proposition \ref{prop-main}.

\subsection{Case $M = S^2$} A function $f:S^2\to \bbR$ has two isolated critical points. The proof of this case is similar to the case of $D^2$. We left details to the reader.

\section{Proof of Theorem \ref{thm:thm-main-2} for cylinder and torus}\label{sec:prof-main-f}
Denote by $B = [0,1]$ or $S^1$.
Let $M$ be a surface diffeomorphic to $S^1\times B$.
Let $f$ be a function from $\FF$ and $f_0:S^1\times B\to B$ be prime function on $S^1\times B$. Denote by $\Delta_f$ and $\Delta_{f_0}$  foliations on $M$ and $S^1\times B$, see Section \ref{sec:hom-prop}.

\subsection*{Step 1. Diffeomorphism $h$} Let $F$ and $F^0$ be normalized $H$-fields on $M$ and $S^1\times B$ for $f$ and $f_0$ with the flows $\bfF_t:M\to M$ and $\bfF^0_t:S^1\times [0,1]\to S^1\times [0,1]$, $t\in \bbR$.
Let also $\phi:M\times S^1\to M$ and $\psi:S^1\times B\times S^1\to S^1\times B$ be two $S^1$-actions on $M$ and $S^1\times B$ given by 
the formula  \eqref{eq:ciecle-act} for $\bfF$ and  $\bfF^0$ respectively. 

Functions $f$ and $f_0$ have no isolated critical points (see By Proposition \ref{prop:properties-f}), so $\phi$  and $\psi$ are free $S^1$-actions. In particular $\psi = \sst$, see Example \ref{ex:circle-act-primitive}.  By Lemma \ref{lm:free-S1-action} actions $\phi$ and $\psi$ are conjugated by some diffeomorphism $h:S^1\times B\to M$, i.e., $\phi_a = h\circ \psi_a\circ h^{-1}$ for all $a\in S^1.$ Such $h$ is a {\it foliated diffeomorphism}, i.e., it maps leaves of a foliation $\Delta_{f_0}$ on $S^1\times B$ to leaves of $\Delta_f$ on $M$, and it is one of the ingredients of the decomposition \eqref{eq:f-funct}.

\subsection*{Step 2. Function $\varkappa$} 
Denote by $L= \{ (1,b)\,|\, b\in B\}\subset S^1\times B.$ A submanifold $L$ transversally intersects each leaf of $\Delta_{f_0}$ at exactly one point, so $L$ is globally transversal to $\Delta_{f_0}$.  Since $f_0:S^1\times B\to B$ has no critical points, it follows that $f_0|_L:L\to B$ is a diffeomorphism, with the inverse $f^{-1}_0(b) = (1,b).$

The image $h(L)\subset M$ is also  globally transversal to $\Delta_f$ since $h$ is a foliated diffeomorphism from Step 1.
Define a function $\varkappa:B\to P$ by the rule:
\begin{equation}\label{eq:varkappa}
	\varkappa = f|_{h(L)}\circ h|_{L}\circ (f_0|_{L})^{-1},
\end{equation}
or by an explicit formula: $\varkappa(s) = f(h(1,s))$ for $s\in B.$ 
\subsection*{Step 3. Verifying {\rm (A)} and {\rm (B)}}
For each $(z,t)\in S^1\times B$ we have $(\varkappa\circ f_0) (z,b) = f(h(1,b))$ and from the other hand $f(h(1,b)) = f(h(z,b))$ since $h$ is a foliated diffeomorphism. This means that $f\circ h = \varkappa\circ f_0$, i.e, $f = \varkappa\circ f_0\circ h^{-1}$, and so a function $f$ decomposes as \eqref{eq:f-funct}, where $h$ and $\varkappa$ are as above.

From a chain rule  if $C$ is a critical circle of $f$, then $f_0(h^{-1}(C))\in B$ is a critical point of $\varkappa.$ So there is a bijection of the set of critical points of $\varkappa$ and the set of critical circles of $f.$
In particular from (3b) of Definition \ref{def:class-FF}  follows that  $\varkappa$ is not flat at each of its critical point.

If $f$ does not have critical circles, then $\varkappa'(t)\neq 0$ for all $t\in B$, so $\varkappa$ is a diffeomorphism. This completes the proof of Theorem \ref{thm:thm-main-2} for the case of $M$ diffeomorphic to  $S^1\times [0,1]$ or $T^2$.\qed

\section{Axially lemmas}\label{sec:Axially}
In this section, we will present several statements and corollaries that will be needed in our proof of Theorem \ref{thm:thm-main-2} for the case of functions on  $D^2$ and $S^2$.
\subsection{Smooth equivalence for Morse functions from $\FF$}
Let $f_0:M\to B$ be a prime function and $f$ be a Morse function from $\FF$ on  $M_0\cong M$.
The following results is known for specialists.
\begin{lemma}\label{lm:smooth-equiv}
	Let $f\in \FF$ be a Morse function on $M$. Then $f$ is smooth equivalent to a prime function $f_0$, i.e., there exists diffeomorphisms $h:M_0\to M$ and $\ell:B\to P$ such that $f = \ell\circ f_0\circ h^{-1}.$ Consequently any Morse function $g\in \mathcal{F}^{\circ}(N,P)$ on $N\cong M$ is smoothly equivalent to $f$.
\end{lemma}
\begin{proof}
(Sketch) If $M$ is diffeomorphic to $S^1\times [0,1]$ or $T^2$ then $f$ has no critical points and Lemma \ref{lm:smooth-equiv} follows from Lemma \ref{lm:free-S1-action} or from Theorem \ref{thm:thm-main-2}.

If $M$ is diffeomorphic to $D^2$ or $S^2$, then the function $f$ has, respectively, one or two Morse extrema. 
There is a diffeomorphism $h$ mapping neighborhood(s) of  extreme(s) of $f$ to the corresponding extreme(s) of a prime functions $f_0$. Such diffeomorphism $h$ is defined by charting maps, see Definition \ref{def:class-F}. Then $h$ can be extended to a diffeomorphism of $M$ by using trajectories of gradient fields of $f$ and $f_0$, see e.g., proof of  Theorem 2.31 \cite{Matsumoto:2002}.	
\end{proof}

\subsection{One corollary of Lemma \ref{lm:h-conj-id} for functions on cylinder} \label{subsec:One-corollary}
Let $M,N$ be surfaces diffeomorphic to $M_0 = S^1\times [0,1]$, $f\in \FF$, $g\in \mathcal{F}^{\circ}(N,P)$ be a Morse function, and $f_0:M_0\to [0,1]$ be a prime function. 
The following result is a simple corollary of Theorem \ref{thm:thm-main-2}.
\begin{corollary}\label{cor:conj-cylinder}
	For a function $f$  there exist a diffeomorphism $h:N\to M$ and a smooth function $\alpha:\mathrm{Im}(g)\to P$ satisfying {\rm (A)} and {\rm (B)} and such that $f = \alpha\circ g\circ h^{-1}$.
\end{corollary}
\begin{proof}
By Theorem \ref{thm:thm-main-2} we have $f = \varkappa_0\circ f_0\circ h_0^{-1}$, where $h_0:M_0\to M$ is a diffeomorphism and $\varkappa_0:[0,1]\to P$ is a smooth function satisfying (A) and (B).  From the other hand, by Lemma \ref{lm:smooth-equiv} we have $g = \ell\circ f_0\circ h_1^{-1}$ for some diffeomorphisms $h_1:M_0\to N$ and $\ell:[0,1]\to P$. Expressing $f_0$ from the last formula we obtain:

\begin{align}
	f &= \varkappa_0\circ f_0\circ h_0^{-1} \nonumber \\ 
	&=\varkappa_0\circ (\ell^{-1}\circ g\circ h_1)\circ h_0^{-1} \nonumber\\
	&= (\varkappa_0\circ \ell^{-1})\circ g_0\circ (h_0\circ h_1^{-1})^{-1} =  \alpha\circ g\circ h^{-1}, \label{eq:compos-function}
\end{align}
where $\alpha = \varkappa_0\circ \ell^{-1}$ and $h = h_0\circ h_1^{-1}$ Obviously that $\alpha$ satisfies (A) and (B).
\end{proof}

Let also $L= \{ (1,b)\,|\, b\in [0,1]\}\subset S^1\times [0,1]$ be such in Step 2. Then applying Theorem \ref{thm:thm-main-2} proved for the case of cylinder for $f$ and $g$  on can obtain
\begin{equation}\label{eq:alpha}
\alpha = \varkappa_0\circ \ell^{-1} = f|_{h_0(L)}\circ (h_0|_L\circ h_1^{-1}|_L) \circ g^{-1}|_{h_1(L)}.
\end{equation}

The following lemma is our main proof tool for the proof of  Theorem \ref{thm:thm-main-2} for functions on on disks and spheres.
\begin{lemma}\label{cor:main}
	Let $M$ be a surface diffeomorphic to $S^1\times [0,1],$
	$f,g$ be functions from $\mathcal{F}^{\circ}(M,\bbR)$ such that $g$ is Morse function. 	
	Suppose that $f$ is decomposed as
	 $f = \alpha\circ g\circ h^{-1}$ for some diffeomorphism $h:M\to M$ and a smooth function $\alpha:g(M)\to P$ satisfying {\rm(A)} and {\rm (B)}, see {\rm Corollary \ref{cor:conj-cylinder}}.
	Let also $W, V$ be open connected $f$-foliated  subsets of $M$ such that $\overline{V}\subset W$.
	Assume that $f = g$ on $\overline{W}$.
	Then
	\begin{itemize}
		\item a diffeomorphism $h$ can be isotoped to a diffeomorphism $\tilde{h}:M\to M$ such that $\tilde{h} = \id$ on $\overline{V}$ and $\tilde{h} = h$ on $\overline{M\setminus W}$, and 	$f = \tilde{\alpha}\circ g\circ {\tilde{h}}^{-1}$ for some smooth function $\tilde{\alpha}:g(M)\to P$ satisfying {\rm (A)} and {\rm (B)}, and
		\item $\tilde{\alpha}(t) = t$ for $t\in g({V})$.
	\end{itemize}
\end{lemma}
\begin{proof}
	Assume that a smooth function $\alpha$ and a diffeomorphism $h$  such that $f = \alpha\circ g\circ h^{-1}$ are  defined by the formula \eqref{eq:compos-function}.
	
	Let $\phi,\psi:M\times S^1\to M$ be free $S^1$-actions induced by functions $f$ and $g$ by their normalized $H$-fields, see \eqref{eq:ciecle-act}.
	Since $f = g$ coincide on $\overline{W}$, it follows that $\phi_a = \psi_a$ on $\overline{W}$ for all $a\in S^1.$ A needed statement is the consequence of  Lemma \ref{lm:h-conj-id}, i.e., $h$ can be isotoped to a diffeomorphism $\tilde{h}:M\to M$ via an isotopy $H_s:M\to M$  given by \eqref{eq:H-isotopy} between $H_0 = h$ and $H_1 = \tilde{h}$ satisfying  $\tilde{h} = \id$ on $\overline{V}$ and $\tilde{h} = h$ on $\overline{M\setminus W}$. 
	
	An isotopy $H_s$ given by \eqref{eq:H-isotopy} induces an isotopy $\tilde{H}_s:S^1\times [0,1]\to M$  given by $\tilde{H}_s = H_s\circ h_1$ between $\tilde{H}_0  = h_0\circ h_1^{-1} \circ h_1 = h_0$ and $\tilde{H}_1 = h_1$.

	 These isotopies defines a path $\alpha:[0,1]\to C^{\infty}(g(M), P)$ by the formula
	\begin{equation}
		\alpha_s = f|_{h_0(L)}\circ H_s|_L\circ \large( g|_{\tilde{H}_{1-s}(L)} \large)^{-1}
	\end{equation}
	between $\alpha_0 = f|_{h_0(L)}\circ h|_L\circ (g|_{h_1(L)})^{-1} = \alpha$, see \eqref{eq:alpha}, and 
	$\alpha_1 = f|_{h_0(L)}\circ \tilde{h}|_L\circ (g|_{h_0(L)})^{-1}$.
	Since $\tilde{h} = \id$ on $\overline{V}$, a function $\alpha_1$ on $g(V)$ has the form
	\begin{align*}
		\alpha_1 &= f|_{h_0(L)}\circ \tilde{h}|_L\circ (g|_{h_0(L)})^{-1}\\
		&=f|_{h_0(L)}\circ (g|_{h_0(L)})^{-1} &\text{($\tilde{h} = \id$ on $\overline{V}$)} \\
		&= \id & \text{(since $f=g$ on $\overline{W}$)} 
	\end{align*}
Obviously that $\alpha_1$ satisfies (A) and (B).
Setting $\tilde{\alpha} = \alpha_1$ we obtain a needed decomposition $f = \tilde{\alpha}\circ g\circ \tilde{h}^{-1}$.	 
\end{proof}

\subsection{Preparation lemma for functions on $D^2$}
Let $M$ be a surface diffeomorphic to $D^2$. Recall that any function $f$ from $\FF$ has a unique non-degenerate critical point $z$, see Proposition \ref{prop:properties-f}. 

The following lemma is a partial case of Theorem \ref{thm:thm-main-2} for the case of $D^2$ and it is needed to simplify the proof of our main result.

\begin{lemma}\label{lm:prep2} Let $M$ be a surface diffeomorphic to $D^2$,
	$f,g$ be a function from $\FF$ such that	
	$g$ is Morse function.  Let also $V$ be open connected and $f$-foliated neighborhood of $z$ such that $\overline{V}$ does not contain critical circles of $f$.  Assume that   $f = g$ on  $\overline{V}$. Then $f = \varkappa\circ g\circ h^{-1}$ for some diffeomorphism $h:M\to M$ and a smooth function $g:\mathrm{Im}(g)\to \bbR$ satisfying {\rm (A)} and {\rm (B)}.
\end{lemma}
\begin{proof}
	Let $U,W$ be $f$-foliated neighborhoods of $z$ diffeomorphic to $D^2$ and such that $\overline{U}\subset W$ and $\overline{W}\subset V$. Denote by $Q = \overline{M\setminus U}$, $L = \overline{M\setminus W}$ and $K = \overline{Q\setminus L}$ subcylinders of $M$. Note that $Q$ contains all critical circles of $f$.
	%
	Restrictions $f|_Q$ and $g|_Q$ belongs to $\mathcal{F}^{\circ}(Q,\bbR)$, and $g$ has no critical points. So by Corollary \ref{cor:conj-cylinder} there exist a diffeomorphism $r:Q\to Q$ and a smooth function $\alpha:g(Q)\to B$ satisfying (A) and (B) such that
	\begin{equation}
		f|_Q = \alpha\circ g|_Q\circ r^{-1}.
	\end{equation}
	Since $f = g$ on $V$, it follows from Lemma \ref{cor:main} that there exist a diffeomorphism $\tilde{r}:Q\to Q$ such that $\tilde{r} = \id$ on $K$, $\tilde{r} = r$ on $M\setminus W$ and $\tilde{\alpha}:g(Q)\to P$ satisfying (A) and (B) and $\tilde{\alpha} = \id$ on $g(K)$ so that $f|_Q = \alpha\circ g|_Q\circ \tilde{r}^{-1}$.
	We  extend $\tilde{r}$ on $U$ and $\tilde{\alpha}$ on $g(U)$ by identities and denote these extensions by $h$ and $\varkappa$. Finally we obtain that $f = \varkappa\circ g\circ {h}^{-1}$. 
	
	Since $\varkappa|_Q$ satisfies (A) and (B), it follows that $\varkappa$ also satisfies them by its definition. 
\end{proof}

\subsection{Approximations of functions from $\FF$ on disk and sphere}
It is knows that any smooth function, so any $f\in \FF$ can be  smoothly approximated by some Morse function from the homotopy class $[f]$. The following lemma is a variant of this statement needed for our proofs.
\begin{lemma}[cf. Theorem 2.20 \cite{Matsumoto:2002}]\label{lm:prep1}
	Let $M$ be surface diffeomorphic to $D^2$ or $S^2$,
	$f$ be a function from $\FF$, $z$ be its isolated critical point, and $V$ is a small open connected $f$-foliated subset of $M$ such that $\overline{V}$ does not contain critical circles of $f$. Then there exists a Morse function $g\in \FF$ such that $f = g$ on $\overline{V}.$
\end{lemma}
\noindent For the sake of completeness, we will recall the procedure how to obtain $g$ from the given $f$ such as above.
\begin{proof}
	(Sketch)
Perturbing $f$ near each critical circle of $f$ sufficiently enough that the resulting function $g_0:M\to \bbR$ is Morse function which coincides with $f$ on $\overline{V},$ i.e., $g_0$ has no critical circles. Such obtained $g_0$ will have some ``new'' non-degenerated extremes and saddles.  This ``new'' saddles and extemums come into pairs.
After some ``rearranging'' critical points (Theorem 3.22 \cite{Matsumoto:2002}) these ``new'' saddles and local extremums can be ``canceled'' by (Theorem 3.28 \cite{Matsumoto:2002}) away from $\overline{V}$ to obtain a needed function  $g$. 
\end{proof}

\section{Proof of Theorem \ref{thm:thm-main-2} for $D^2$ and $S^2$}\label{sec:prof-main-f-2}
\subsection{Proof of Theorem \ref{thm:thm-main-2} for the case of $D^2$}\label{sec:proof-disk-sphere}
Let $M$ be a surface diffeomorphic to $M_0 = D^2$ and $f$ be a function from $\FF$. Denote by $z$ a unique non-degenerated critical point of $f$, see Proposition \ref{prop:properties-f}.

{\it Step 1.} Let $V$ be an open connected and $f$-foliated neighborhood of $z$ such that $\overline{V}$ does not contain critical circles of $f$. By Lemma \ref{lm:prep1} there exists a Morse function $g\in \FF$ such that $f = g$ on $\overline{V}$.

{\it Step 2.} By Lemma \ref{lm:prep2} there exist a diffeomorphism $h:M\to M$ and a smooth function $\alpha:g(M)\to P$ satisfying (A) and (B), and such that 
$f = \alpha\circ g\circ h_0^{-1}$.

{\it Step 3.} By Lemma \ref{lm:smooth-equiv} $g$ is smoothly equivalent  to a prime function $f_0:D^2\to B,$ i.e., $g = \ell\circ f_0\circ h_1^{-1}$ for some diffeomorphisms $h_1:M_0\to M$ and $\ell:B\to P$.

{\it Step 4.} Combining results from Step 2 and Step 3 we obtain 
\begin{align*}
	f &= \alpha\circ g\circ h^{-1}_0\\ &= \alpha\circ (\ell\circ f_0\circ h_1^{-1})\circ h_0^{-1}\\ &=\varkappa\circ f_0\circ h^{-1} 
\end{align*}
where $\varkappa = \alpha\circ \ell$ and $h = h_0\circ h_1$.
\qed

A next fact is a consequence of 
\begin{corollary}\label{cor:4}
	Let $f$ be  function from $\FF$ and $g$ be a Morse function from $\mathcal{F}^{\circ}(N,P)$, where $N$ is diffeomorphic to $M\cong D^2$. 
	\begin{enumerate}
		\item Then $f = \alpha\circ g\circ h^{-1}$ for some diffeomorphism $h:N\to M$ and a smooth function $\alpha: \mathrm{Im}(g)\to P$ satisfying {\rm (A)} and {\rm (B)}.
		\item Assume that $M = N$. Let $W, U$ be open connected and $f$-foliated neighborhood of $\partial M$ such that $\overline{W}\subset U$ and $\overline{U}$ does not contain critical circles of $f$.
		If $f = g$ on $\overline{U}$, then there exist a diffeomorphism $\tilde{h}:M\to M$ and a smooth function $\tilde{\alpha}:g(M)\to P$ such that $f = \tilde{\alpha}\circ g\circ \tilde{h}^{-1}$, where $\tilde{\alpha} = \id$ in $g(\overline{W})$, $\tilde{\alpha} = \alpha$ on $M\setminus W$ and $\tilde{h} = \id$ on $\overline{W}$, $\tilde{h} = h$ on $M\setminus U.$
	\end{enumerate}
\end{corollary}
\begin{proof}
	(1) follows from Theorem \ref{thm:thm-main-2} proved for the case of disks and Lemma \ref{lm:smooth-equiv}, (2) is a consequence of Lemma \ref{cor:main}. We left details to the reader.
\end{proof}

\subsection{Proof of Theorem \ref{thm:thm-main-2} for the case of $S^2$}
Let $M$ be a surface diffeomorphic to $S^2$ and $f$ be a function from $\FF$. By Proposition \ref{prop:properties-f} $f$ has two non-degenerated extremes $z_0$ and $z_1$.

Let $V_i, W_i$ be open connected and $f$-foliated neighborhood of $z_i$ such that $\overline{W}_i\subset V_i$ and $\overline{V}$ does not contain critical circles of $f$, $i = 0,1$.
Let also $Q$ be an open $f$-foliated subset of $M$ containing all critical circles of $f$ and such that $Q\cap W_i\neq \varnothing$, $i = 0,1$. A subspace $Q$ is diffeomorphic to a cylinder.
Denote by $U_1 = Q\cap V_1$.

{\it Step 1.}  By Lemma \ref{lm:prep1} there exists a Morse function $g\in \FF$ such that $f = g$ on $\overline{V}_0$.

{\it Step 2.} Since $U_1$ is diffeomorphic to a disk, it follows from (1) of Corollary \ref{cor:4} that there exists a diffeomorphism $h:U_1\to U_1$ and a smooth function $\alpha_1:g(U_1)\to P$ such that $f|_{U_1} = \alpha_1\circ g|_{U_1}\circ h_1^{-1}.$

{\it Step 3.} Since $f = g$ on $V_0\cap U_1$, which is an connected $f$-foliated neighborhood of $\partial U_1$, then by (2) of Corollary there exist a diffeomorphism $\tilde{h}_1:U_1\to U_1$ and a smooth function $\tilde{\alpha}_1:g(U_1)\to P$ such that $\tilde{h}_1 = \id$ on $W_0\cap U_1$, $\tilde{h}_1 = h$ on $U_1\setminus V_0$, and $\tilde{\alpha}_1 = \id$ on $g(W_0\cap U_1)$, $\tilde{\alpha_1} = \alpha_1$ on $g(U_1\setminus V_0)$ so that $f|_{U_1} = \tilde{\alpha}_1\circ g|_{U_1}\circ \tilde{h}_1^{-1}$

{\it Step 4}. Since $f = g$ on $\overline{V}$, it follows that a diffeomorphism $\tilde{h}_1$ and a smooth function $\tilde{\alpha}_1$ can be extended by identities on $M$ and $g(M)$ respectively; denote this extensions by $\tilde{\alpha}$ and $\tilde{h}$. We obtain
$f = \tilde{\alpha}\circ g\circ \tilde{h}^{-1}.$

{\it Step 5}.  By Lemma \ref{lm:smooth-equiv} $g$ is smoothly equivalent  to $f_0:D^2\to B,$ i.e., $g = \ell\circ f_0\circ r^{-1}$ for some diffeomorphisms $r:M_0\to M$ and $\ell:B\to P$. 

{\it Step 6.}  Combining results from Step 4 and Step 5 we obtain 
\begin{align*}
	f &= \tilde{\alpha}\circ g\circ \tilde{h}^{-1}\\ &= \tilde{\alpha}\circ (\ell\circ f_0\circ r^{-1})\circ \tilde{h}^{-1}\\ &=\varkappa\circ f_0\circ h^{-1} 
\end{align*}
where $\varkappa = \tilde{\alpha}\circ \ell$ and $h = \tilde{h}\circ r$.
\qed

\bibliographystyle{plain}

\begin{thebibliography}{10}
	
	\bibitem{BatistaCostaMeza:2022}
	E.~B. Batista, J.~C.~F. Costa, and I.~S. Meza-Sarmiento.
	\newblock Topological classification of circle-valued simple {M}orse-{B}ott
	functions on closed orientable surfaces.
	\newblock {\em J. Singul.}, 25:78--89, 2022.
	
	\bibitem{Feshchenko:Arxiv2023}
	Bohdan Feshchenko.
	\newblock Homotopy type of stabilizers of circle-valued functions with
	non-isolated singularities on surfaces.
	\newblock {\em arXiv:2305.08255}, page~9, 2023.
	
	\bibitem{Kadubovskiy:UMZH:2006}
	O.~A. Kadubovskiy.
	\newblock Topological equivalence of functions on oriented surfaces.
	\newblock {\em Ukra\"in. Mat. Zh.}, 58(3):343--351, 2006.
	
	\bibitem{Koszul:GD:1953}
	J.~L. Koszul.
	\newblock Sur certains groupes de transformations de {L}ie.
	\newblock In {\em G\'eom\'etrie diff\'erentielle. {C}olloques {I}nternationaux
		du {C}entre {N}ational de la {R}echerche {S}cientifique, {S}trasbourg, 1953},
	pages 137--141. CNRS, Paris, 1953.
	
	\bibitem{Kulinich:1998:MFAT}
	E.~V. Kulinich.
	\newblock On topologically equivalent {M}orse functions on surfaces.
	\newblock {\em Methods Funct. Anal. Topology}, 4(1):59--64, 1998.
	
	\bibitem{Maksymenko:Nel:2009}
	S.~I. Maksymenko.
	\newblock Symmetries of center singularities of plane vector fields.
	\newblock {\em Nel\=\i n\=\i\u in\=\i\ Koliv.}, 13(2):177--205, 2010.
	
	\bibitem{Maksymenko:BulletinSciMath:2006}
	Sergey Maksymenko.
	\newblock Stabilizers and orbits of smooth functions.
	\newblock {\em Bull. Sci. Math.}, 130(4):279--311, 2006.
	
	\bibitem{Maksymenko:TA:2003}
	Sergiy Maksymenko.
	\newblock Smooth shifts along trajectories of flows.
	\newblock {\em Topology Appl.}, 130(2):183--204, 2003.
	
	\bibitem{Maksymenko:AGAG:2006}
	Sergiy Maksymenko.
	\newblock Homotopy types of stabilizers and orbits of {M}orse functions on
	surfaces.
	\newblock {\em Ann. Global Anal. Geom.}, 29(3):241--285, 2006.
	
	\bibitem{Maks:reparam-sh-map}
	Sergiy Maksymenko.
	\newblock Reparametrization of vector fields and their shift maps.
	\newblock {\em Pr. Inst. Mat. Nats. Akad. Nauk Ukr. Mat. Zastos.},
	6(2):489--498, arXiv:math/0907.0354, 2009.
	
	\bibitem{Maksymenko:UMZ:ENG:2012}
	Sergiy Maksymenko.
	\newblock Homotopy types of right stabilizers and orbits of smooth functions
	functions on surfaces.
	\newblock {\em Ukrainian Math. Journal}, 64(9):1186--1203, 2012.
	
	\bibitem{Maksymenko:2021:review}
	Sergiy Maksymenko.
	\newblock Deformations of functions on surfaces.
	\newblock {\em Proceedings of Institute of Mathematics of NAS of Ukraine},
	17(2):150--199, 2020.
	
	\bibitem{MartinezMezaOliveria:2016}
	J.~Mart\'{\i}nez-Alfaro, I.~S. Meza-Sarmiento, and R.~Oliveira.
	\newblock Topological classification of simple {M}orse {B}ott functions on
	surfaces.
	\newblock In {\em Real and complex singularities}, volume 675 of {\em Contemp.
		Math.}, pages 165--179. Amer. Math. Soc., Providence, RI, 2016.
	
	\bibitem{Matsumoto:2002}
	Yukio Matsumoto.
	\newblock {\em An introduction to {M}orse theory}, volume 208 of {\em
		Translations of Mathematical Monographs}.
	\newblock American Mathematical Society, Providence, RI, 2002.
	\newblock Translated from the 1997 Japanese original by Kiki Hudson and
	Masahico Saito, Iwanami Series in Modern Mathematics.
	
	\bibitem{Morita:TMM:2001}
	Shigeyuki Morita.
	\newblock {\em Geometry of differential forms}, volume 201 of {\em Translations
		of Mathematical Monographs}.
	\newblock American Mathematical Society, Providence, RI, 2001.
	\newblock Translated from the two-volume Japanese original (1997, 1998) by
	Teruko Nagase and Katsumi Nomizu, Iwanami Series in Modern Mathematics.
	
	\bibitem{Prishlyak:TA:2002}
	A.~O. Prishlyak.
	\newblock Topological equivalence of smooth functions with isolated critical
	points on a closed surface.
	\newblock {\em Topology Appl.}, 119(3):257--267, 2002.
	
	\bibitem{Sharko:UMZ:2003}
	V.~V. Sharko.
	\newblock Smooth and topological equivalence of functions on surfaces.
	\newblock {\em Ukra\"\i n. Mat. Zh.}, 55(5):687--700, 2003.
	
	\bibitem{Tu:EquivariantCohomology:2020}
	Loring~W. Tu.
	\newblock {\em Introductory lectures on equivariant cohomology}, volume 204 of
	{\em Annals of Mathematics Studies}.
	\newblock Princeton University Press, Princeton, NJ, 2020.
	\newblock With appendices by Tu and Alberto Arabia.
	
	\bibitem{Whitney:DukeMathJ:1943}
	Hassler Whitney.
	\newblock Differentiable even functions.
	\newblock {\em Duke Math. J.}, 10:159--160, 1943.
	
\end{thebibliography}

\end{document}